\documentclass[12pt,a4paper]{article}
\usepackage{amssymb,amsmath,amsthm}

\newcommand\cD{{\mathcal D}}

\newcommand\cF{{\mathcal F}}
\newcommand\cG{{\mathcal G}}

\newcommand\cM{{\mathcal M}}
\newcommand\cN{{\mathcal N}}
\newcommand\cP{{\mathcal P}}

\theoremstyle{plain}
\newtheorem{theorem}{Theorem}[section]
\newtheorem{lemma}[theorem]{Lemma}

\newtheorem{conjecture}[theorem]{Conjecture}
\newtheorem{proposition}[theorem]{Proposition}

\theoremstyle{definition}

\newtheorem{claim}[theorem]{Claim}

\newcommand\lref[1]{Lemma~\ref{lem:#1}}
\newcommand\tref[1]{Theorem~\ref{thm:#1}}
\newcommand\cref[1]{Corollary~\ref{cor:#1}}
\newcommand\clref[1]{Claim~\ref{clm:#1}}

\newcommand\pref[1]{Proposition~\ref{prop:#1}}

\title{Game saturation of intersecting families}

\author{Bal\'azs Patk\'os\thanks{Alfr\'ed R\'enyi Institute of Mathematics, P.O.B. 127, Budapest H-1364, Hungary. Email: patkos.balazs@renyi.mta.hu. Research supported by
    Hungarian National Scientific Fund, grant number: PD-83586 and the J\'anos Bolyai Research Scholarship of the Hungarian Academy of Sciences.} \and M\'at\'e Vizer\thanks{Alfr\'ed R\'enyi Institute of Mathematics, P.O.B. 127, Budapest H-1364, Hungary. Email: vizer.mate@renyi.mta.hu}}

\begin{document}
\maketitle

\begin{abstract}
We consider the following combinatorial game: two players, Fast and Slow, claim $k$-element subsets of $[n]=\{1,2,...,n\}$ alternately, one at each turn, such that both players are allowed to pick sets that intersect all previously claimed subsets. The game ends when there does not exist any unclaimed $k$-subset that meets all already claimed sets. The score of the game is the number of sets claimed by the two players, the aim of Fast is to keep the score as low as possible, while the aim of Slow is to postpone the game's end as long as possible. The game saturation number is the score of the game when both players play according to an optimal strategy. To be precise we have to distinguish two cases depending on which player takes the first move. Let $gsat_F(\mathbb{I}_{n,k})$ and $gsat_S(\mathbb{I}_{n,k})$ denote the score of the saturation game $(X,\cD)$ when both players play according to an optimal strategy and the game starts with Fast's or Slow's move, respectively. We prove that  $\Omega_k(n^{k/3-5}) \le gsat_F(\mathbb{I}_{n,k}),gsat_S(\mathbb{I}_{n,k}) \le O_k(n^{k-\sqrt{k}/2})$ holds.
\end{abstract}


\medskip

\section{Introduction}

\

A very much studied notion of extremal combinatorics is that of saturation. Let $\cP$ be a property of hypergraphs such that whenever a hypergraph $H$ possesses property $\cP$, then so do all subhypergraphs of $\cP$. We say that the hypergraph $H$ is \textit{saturated} with respect to $\cP$ if $H$ has property $\cP$, but for any hyperedge $e \notin E(H)$ the hypergraph $H+e$ does not have property $\cP$ anymore.
A typical problem in extremal combinatorics is to determine $ex(n,\cP)$ ($sat(n,\cP)$) the most (least) number of hyperedges that a hypergraph on $n$
vertices may contain provided it is saturated with respect to $\cP$.

Lots of combinatorial problems have their game theoretical analogs. For a survey on combinatorial games see Fraenkel's paper \cite{Fr}. For topics focusing on positional games, we refer the reader to the book of Beck \cite{B} and the forthcoming book of Hefetz, Krivelevich, Stojakovi\'c and Szab\'o \cite{HKSS}. There are two types of combinatorial games that are related to saturation problems. One of them originates from Hajnal's triangle game \cite{S,MS}, but its more general form is as follows \cite{F,FHJ}: given a family $\cF$ of excluded subgraphs and a host graph $G$, two players pick the edges of $G$ alternately such that the set of all claimed edges should form an $\cF$-free subgraph $H$ of $G$ (i.e. no $F \in \cF$ occurs as a subgraph in $H$). Whenever $H$ becomes $\cF$-saturated, the player on turn cannot make a move and loses/wins (depending on the rules of the game). In Hajnal's triangle game the family $\cF$ consists only of the triangle graph.

In this paper we study a game played by two players, Fast and Slow such that Fast's aim is to create a maximal hypergraph the size of which is as close to the saturation number as possible while Slow's aim is to create a maximal hypergraph the size of which is as close to the extremal number as possible.
More formally, the \textit{saturation game} $(X,\cD)$ is played on the \textit{board} $X$ according to the \textit{rule} $\cD \subseteq 2^{X}$, where $\cD$ is downward closed (or decreasing) family of subsets of $X$, that is $E \subset D \in \cD$ implies $E \in \cD$. Two players Fast and Slow pick one unclaimed element of the board at each turn alternately such that at any time $i$ during the game, the set $C_i$ of all elements claimed thus far belongs to $\cD$. The elements $x \in X \setminus C_i$ for which $\{x\} \cup C_i \in \cD$ holds will be called the \textit{legal moves} at time $i+1$ as these are the elements of the board that can be claimed by the player on turn. The game ends when there is no more legal moves, that is when $C_i$ is a maximal set in $\cD$ and the score of the game is the size of $C_i$. The aim of Fast is to finish the game as fast as possible and thus obtain a score as low as possible while the aim of Slow is to keep the game going as long as possible. The \textit{game saturation number} is the score of the game when both players play according to an optimal strategy. To be precise we have to distinguish two cases depending on which player takes the first move. Let $gsat_F(\cD)$ and $gsat_S(\cD)$ denote the score of the saturation game $(X,\cD)$ when both players play according to an optimal strategy and the game starts with Fast's or Slow's move, respectively. In most cases, the board $X$ is either $\binom{[n]}{k}$ for some $1 \le k \le n$ or $2^{[n]}$. Clearly, the inequalities $sat(\cD) \le gsat_F(\cD), gsat_S(\cD)\le ex(\cD)$ hold.

The first result concerning saturation games is due to F\"uredi, Reimer and Seress \cite{FRS}. They considered the case when the board $X$ is the edge set of the complete graph on $n$ vertices and $\cD=\cD_{n,K_3}$ is the family of all triangle-free subgraphs of $K_n$. They established the lower bound $\frac{1}{2}n\log n\le gsat_F(\cD_{n,K_3}), gsat_S(\cD_{n,K_3})$ and claimed without proof an upper bound $\frac{n^2}{5}$ via personal communication with Paul Erd\H os. Their paper mentions that the first step of Fast's strategy is to build a $C_5$-factor. However, as it was recently pointed out by Hefetz, Krivelevich and Stojakovi\'c \cite{HKS}, Slow can prevent this to happen. Indeed, in his first $\lfloor \frac{n-1}{2}\rfloor$ moves, Slow can create a vertex $x$ with degree $\lfloor \frac{n-1}{2}\rfloor$, and because of the triangle-free property, the neighborhood of $x$ must remain an independent set throughout the game. But clearly, a graph that contains a $C_5$-factor cannot have an independent set larger than $2n/5+4$.

Recently, Cranston, Kinnersley, O and West \cite{CKOW} considered the saturation game when the board $X=X_G$ is the edge set of a graph $G$ and $\cD=\cD_G$ consists of all (partial) matchings of $G$.

In this paper, we will be interested in intersecting families. That is the board $X=X_{n,k}$ will be the edge-set of the complete $k$-graph on $n$ vertices and $\cD=\mathbb{I}_{n,k}$ is the set of intersecting families that is $\mathbb{I}_{n,k}:=\{\cF\subseteq X_{n,k}:  F \cap G\neq \emptyset $ $ \forall F,G \in \cF\}$. Note that by the celebrated theorem of Erd\H os, Ko and Rado \cite{EKR} we have $ex(\mathbb{I}_{n,k})=\binom{n-1}{k-1}$ provided $2k \le n$. The saturation number $sat(\mathbb{I}_{n,k})$ is not known. J-C. Meyer \cite{M} conjectured this to be $k^2-k+1$ whenever a projective plane of order $k-1$ exists. This was disproved by F\"uredi \cite{F} by constructing a maximal intersecting family of size $3k^2/4$ provided a projective plane of order $k/2$ exists, and this upper bound was later improved by Boros, F\"uredi, and Kahn to $k^2/2+O(k)$ \cite{BFK} provided a projective plane of order $k-1$ exists.
The best known lower bound on $sat(\mathbb{I}_{n,k})$ is $3k$ due to Dow, Drake, F\"uredi, and Larson \cite{DDFL}. This holds for all values of $k$.

We mentioned earlier that the game saturation number might depend on which player starts the game. This is indeed the case for intersecting families. If $k=2$, then after the first two moves the already claimed edges are two sides of a triangle. Thus if Fast is the next to move, he can claim the last edge of this triangle and the game is finished, thus $gsat_F(\mathbb{I}_{n,2})=3$. On the other hand, if Slow can claim the third edge, then he can pick an edge containing the intersection point of the first two edges and then all such edges will be claimed one by one and we obtain $gsat_S(\mathbb{I}_{n,2})=n-1$.

The main result of the present paper is the following theorem that bounds away $gsat_S(\mathbb{I}_{n,k})$ and $gsat_F(\mathbb{I}_{n,k})$ both from $ex(\mathbb{I}_{n,k})$ and $sat(\mathbb{I}_{n,k})$ if $n$ is large enough compared to $k$.

\vskip 1cm

\begin{theorem}
\label{thm:main} For all $k \ge 2$ the following holds:
\[
\Omega_k\left(n^{\lfloor k/3\rfloor-5}\right) \le gsat_F(\mathbb{I}_{n,k}), gsat_S(\mathbb{I}_{n,k}) \le O_k\left(n^{k-\sqrt{k}/2}\right)
.\]

\end{theorem}

\vskip 2cm

\section{Proof of \tref{main}}

We start this section by defining an auxiliary game that will enable us to prove \tref{main}.

\vspace{3mm}

We say that a set $S$ \textit{covers} a family $\cF$ of sets if $S \cap F \neq \emptyset$ holds for every set $F \in \cF$. The \textit{covering number} $\tau(\cF)$ is the minimum size of a set $S$ that covers $\cF$. Note that if $\cF$ is an intersecting family of $k$-sets, then $\tau(\cF) \le k$ holds as by the intersecting property any set $F \in \cF$ covers $\cF$. The following proposition is folklore, but for the sake of self-containedness we present its proof.

\begin{proposition}
\label{prop:tau} If $\cF \subseteq \binom{[n]}{k}$ is a maximal intersecting family with covering number $\tau$, then the following inequalities hold:
\[
\binom{n-\tau}{k-\tau} \le |\cF| \le k^{\tau }\binom{n-\tau}{k-\tau}.
\]
\end{proposition}

\begin{proof}
The first inequality follows from the following observation: if $S$ covers $\cF$, then all $k$-subset of $[n]$ that contain $S$ must belong to $\cF$ by maximality.

To obtain the second inequality note that denoting the \textit{maximal degree} of $\cF$ by $\Delta(\cF)$ the inequality $|\cF| \le k \Delta(\cF)$ holds. Indeed, by the intersecting property we have $|\cF| \le \sum_{x \in F}d(x)$ for any set $F \in \cF$ and the right hand side is clearly not more than $k\Delta(\cF)$.
Let $d_j(\cF)$ denote the maximum number of sets in $\cF$ that contain the same $j$-subset, thus $d_1(\cF)=\Delta(\cF)$ holds by definition. For any $j<\tau$ and $j$-subset $J$ that is contained in some $F \in \cF$ there exists an $F' \in \cF$ with $J \cap F'=\emptyset$. Thus $d_j(\cF) \le kd_{j+1}(\cF)$ is true. Since $d_\tau(\cF)\le \binom{n-\tau}{k-\tau}$ holds, we obtain $d_1(\cF) \le k^{\tau-1}\binom{n-\tau}{k-\tau}$ and thus the second inequality follows by the first observation of this paragraph.
\end{proof}

The main message of \pref{tau} is that if $k$ is fixed and $n$ tends to infinity, then the order of magnitude of the size of a maximal intersecting family $\cF$ is determined by its covering number. As $|\cF|=\Theta_k(n^{k-\tau(\cF)})$ holds, a strategy in the saturation game that maximizes $\tau(\cF)$ is optimal for Fast up to a constant factor, and a strategy that minimizes $\tau(\cF)$ is optimal for Slow up to a constant factor.

Therefore from now on we will consider the \textit{$\tau$-game} in which two players: \textit{minimizer} and \textit{Maximizer} take unclaimed elements of $X=X_{n,k}=\binom{[n]}{k}$ alternately such that at any time during the game the set of all claimed elements should form an intersecting family. The game stops when the claimed elements form a maximal intersecting family $\cF$. The score of the game is the covering number $\tau(\cF)$ and the aim of \textit{minimizer} is to keep the score as low as possible while \textit{Maximizer}'s aim is to reach a score as high as possible. Let $\tau_m(n,k)$ (resp. $\tau_M(n,k)$) denote the score of the game when both players play according to their optimal strategy and the first move is taken by \textit{minimizer} (resp. \textit{Maximizer}). The following simple observation will be used to define strategies.

\begin{proposition}
\label{prop:minmax} Let $\cG=\{G_1,...,G_{k+1}\}$ be an intersecting family of $k$-sets. Assume that there exists a set $C$ such that $G_1 \setminus C,...,G_{k+1}\setminus C$ are non-empty and pairwise disjoint.
Then we have $\tau(\cF) \le |C|$ for any intersecting family $\cF \supseteq \cG$ of $k$-sets.

\end{proposition}

\begin{proof}
To see the statement, observe that if a $k$-set $F$ is disjoint from $C$, then it cannot meet all $k+1$ of the sets $G_1\setminus C,G_2 \setminus C,...,G_k\setminus C$.
\end{proof}

\tref{main} will follow from the following two lemmas and \pref{tau}.

\begin{lemma}
\label{lem:taumin} For any positive integer $k$, there exists $n_0=n_0(k)$ such that if $n \ge n_0$, then $$\tau_m(n,k),\tau_M(n,k) \le \lceil 2k/3\rceil +4$$ holds.
\end{lemma}

\begin{proof}
We have to provide a strategy for \textit{minimizer} that ensures the covering number of the resulting family to be small. Let us first assume that \textit{minimizer} starts the game and let $m_0,M_1,m_1,M_2,m_2,...$ denote the $k$-sets claimed during the game. \textit{Minimizer}'s strategy will involve sets $A_i,C_i$ for $2 \le i \le k$ with the properties:

\vspace{3mm}

a) $A_i \subseteq m_0$, $C_{i-1} \subseteq C_i$, $|C_i| \le |C_{i-1}|+1$;

\vspace{1mm}

b) the sets $m_0\setminus (A_i \cup C_i)$ and $m_1\setminus C_i,...,m_i \setminus C_i$ are non-empty and pairwise disjoint;

\vspace{1mm}

c) $A_i \cup C_i$ meets all sets $m_j,M_j$ for $j \le i$;

\vspace{1mm}

d) $C_i$ meets all $m_j$'s and all but at most one of the $M_j$'s for $j \le i$.

\vskip 0.3truecm

Before proving how \textit{minimizer} is able to pick her $k$-sets $m_0,m_1,...,m_k$ with the above properties, let us explain why it is good for her. She would like to utilize \pref{minmax} to claim that no matter how the players continue to play, after choosing $m_{k+1}$, she can be sure that the resulting maximal intersecting family will have low covering number. As she cannot control \textit{Maximizer}'s moves, she will apply \pref{minmax} with $k+1$ of her own sets playing the role of $\cG$. The set $C_i$ will be a temporary approximation of a future covering set $C$: it meets all previously claimed sets but at most one and its union with the auxiliary set $A_i$ does indeed meet all sets of the game until round $i$. Whenever \textit{minimizer} decides that an element $x$ is included in $C_i$, then $x$ stays there forever. This is condition a) saying $C_{i-1} \subseteq C_i$. Such a strategy would be easy to follow without the sets $A_i$ and still fulfilling the second part of condition a), namely that the covering set can have at most one new element in each round. Indeed, \textit{minimizer} in the $(i+1)$st round could claim a set $m_{i+1}$ containing $C_i$ and an element $x$ from \textit{Maximizer}'s last move $M_i$ and let $C_{i+1}=C_i\cup \{x\}$. This would be a legal move as $m_{i+1}$ meets all sets $m_0, M_1,m_1,...,M_{i-1},m_i$ as $m_{i+1}$ contains $C_i$ and meets $M_i$ as they both contain $x$. The problem with this strategy is that $C_i$ might grow in each round and the final covering set might be of size $k$.

At the end of the proof of \lref{taumin} we show how the auxiliary sets $A_i$ can help so that in many rounds the covering set does not need to grow at all.
But first we make sure that \textit{minimizer} is able to claim $k$-sets $m_0,m_1,...,m_k$ such that sets $A_i,C_i$ with the above properties exist. \textit{Minimizer} can claim an arbitrary $m_0$, and after \textit{Maximizer}'s first move $M_1$, he can pick $a_1 \in m_0 \cap M_1$ and claim $m_1:=\{a_1\}\cup N_1$ where $N_1$ is a $(k-1)$-set disjoint from $m_0\cup M_1$. \textit{Minimizer}'s strategy distinguishes two cases for claiming $m_2$ depending on \textit{Maximizer}'s second move $M_2$. If $a_1 \in M_2$, then \textit{minimizer} claims $m_2:=\{a_1\} \cup N_2$ with $N_2$ being a $(k-1)$-set disjoint from $m_0 \cup m_1 \cup M_1 \cup M_2$ and we define $A_2:=\emptyset, C_2:=\{a_1\}$. If $a_1 \notin M_2$, then by the intersecting property there exists $c_2 \in M_1 \cap M_2$. Let \textit{minimizer} claim $m_2:=\{a_1,c_2\} \cup N_2$ with $N_2$ being a $(k-2)$-set disjoint from $m_0 \cup m_1 \cup M_1 \cup M_2$ and put $A_2:=\emptyset,C_2:=\{a_1,c_2\}$. In both cases, the properties a)-d) hold.

Let us assume that \textit{minimizer} is able to claim $k$-sets $m_0,m_1,...,m_{i-1}$ and define sets $A_i,C_i$. The strategy of \textit{minimizer} will distinguish several cases depending on \textit{Maximizer}'s move $M_i$. In all cases \textit{minimizer}'s set $m_i$ will consist of elements of $C_{i-1}$, a possible element $a_i$ and elements of a set $N_i$ that is disjoint from all previously claimed sets. As we are interested in not more than $2(k+2)$ sets, therefore \textit{minimizer} will always be able to choose $N_i$ if $n \ge 2k(k+2)$ holds.

\vskip 0.6truecm

\textsc{Case I}: $C_{i-1}$ meets all previously claimed $k$-sets.

\vskip 4mm

$\bullet$ If $M_i \cap C_{i-1}\neq \emptyset$, then

\vskip 2mm

\hskip 1cm $\circ$ let $m_i:=C_{i-1} \cup N_i$ with $|N_i|=k-|C_{i-1}|$ and $N_i \cap (\cup_{j=0}^{i-1}m_j\cup\cup_{j=1}^iM_j)=\emptyset$;

\vskip 1mm

\hskip 1cm $\circ$ let $C_i:=C_{i-1}, A_i:=A_{i-1}$.

\vskip 2mm

\hskip 1cm The set $m_i$ is a legal move for \textit{minimizer} in this subcase
 as it contains $C_{i-1}$.

\hskip 1cm Now observe that

\vskip 2mm

\hskip 1cm a) is satisfied as it was satisfied in step $i-1$;

\vskip 1mm

\hskip 1cm b) is satisfied as it was satisfied in the $i-1$ and by the choice of $N_i$;

\vskip 1mm

\hskip 1cm c) is satisfied as it was satisfied in step  $i-1$ and by the fact that we are in the

\hskip 15mm $M_i \cap C_{i-1} \neq \emptyset$ subcase and we chose $m_i$ to contain $C_{i-1}$;

\vskip 1mm

\hskip 1cm d) is satisfied by the assumptions that $C_i=C_{i-1}$ meets all previously claimed sets

\hskip 15mm and that $C_{i-1}\cap M_{i-1}\neq \emptyset$.

\vskip 0.3cm

Note that if step $i$ is in this subcase, then in step $i+1$ we are still in \textsc{Case I}.

\vskip 1cm

$\bullet$ If $M_i \cap C_{i-1}= \emptyset$ and there exists $a_i \in (M_i \cap m_0)\setminus A_{i-1}$, then

\vskip 2mm

\hskip 1cm $\circ$ let $m_i:=C_{i-1}\cup \{a_i\}\cup N_i$  with $|N_i|=k-|C_{i-1}|-1$ and $N_i \cap (\cup_{j=0}^{i-1}m_j\cup\cup_{j=1}^iM_j)=\emptyset$;

\vskip 1mm

\hskip 1cm $\circ$ let $C_i:=C_{i-1}, A_i=A_{i-1}\cup \{a_i\}$.

\vskip 2mm

\hskip 1cm The set $m_i$ is a legal move for \textit{minimizer} in this subcase
as $m_i$ contains $C_{i-1}$

\hskip 1cm and $a_i$. Now observe that

\vskip 2mm

\hskip 1cm a) is satisfied as it was satisfied in step $i-1$ and by the choice $a_i \in m_0$;

\vskip 1mm

\hskip 1cm b) is satisfied as it was satisfied in step $i-1$ and by the choice of $N_i$ and $a_i$;

\vskip 1mm

\hskip 1cm c) is satisfied as it was satisfied in step $i-1$ and by the fact that $a_i \in M_i \cap m_i$;

\vskip 1mm

\hskip 1cm d) is satisfied as $C_{i-1}=C_{i-1}$ meets all previously claimed sets.

\vskip 0.3truecm

Note that if step $i$ is in this subcase, then in step $i+1$ we are not in \textsc{Case I} as $M_i$ is not 

met by $C_{i-1}=C_i$.

\vskip 1cm

$\bullet$ If none of the above subcases of \textsc{Case I} happen, then we must have $M_i \cap C_{i-1}=\emptyset$

\hskip 3mm and $\emptyset \neq M_i \cap m_0 \subseteq A_{i-1}$, as \textit{Maximizer} must pick $M_i$ such that it intersects all

\hskip 3mm previously claimed $k$-sets, in particular it should intersect $m_0$.

\hskip 3mm Let $a \in M_i \cap m_0$ and thus $a \in A_{i-1}$ and

\vskip 2mm

\hskip 1cm $\circ$ let $m_i:=C_{i-1} \cup \{a\} \cup N_i$  with $|N_i|=k-|C_{i-1}|-1$ and $N_i \cap (\cup_{j=0}^{i-1}m_j\cup\cup_{j=1}^iM_j)=\emptyset$;

\vskip 1mm

\hskip 1cm $\circ$ let $A_i:=A_{i-1}\setminus \{a\}, C_i:=C_{i-1} \cup \{a\}$.

\vskip 2mm

\hskip 1cm The set $m_i$ is a legal move for \textit{minimizer} in this subcase
 as $m_i$ contains $C_{i-1}$

\hskip 1cm and $a$. Now observe that

\vskip 2mm

\hskip 1cm a) is satisfied as it was satisfied in step $i-1$;

\vskip 1mm

\hskip 1cm b) is satisfied as it was satisfied in step $i-1$,  $A_{i-1} \cup C_{i-1} = A_i \cup C_i$ and by the

\hskip 14mm  choice of $N_i$ and $a_i$;

\vskip 1mm

\hskip 1cm c) is satisfied as $A_{i-1} \cup C_{i-1}$ meets all sets $m_0,M_1,...M_{i-1}$, and in this subcase we

\hskip 14mm have $A_{i-1} \cup C_{i-1} = A_i \cup C_i$ and $a \in M_i$;

\vskip 1mm

\hskip 1cm d) is satisfied as $C_{i-1} \subset C_i$ meets all previously claimed sets and $a \in m_i\cap C_i$.

\vskip 0.3truecm

Note that if step $i$ is in this subcase, then in step $i+1$ we are still in \textsc{Case I}.

\vskip 1cm

\textsc{Case II}: There exists an $M_j$ ($j\le i-1$) with $M_j \cap C_{i-1}=\emptyset$.

\vskip 4mm

$\bullet$ As \textit{Maximizer} picks $M_i$ such that it meets all previously claimed $k$-sets, there must

\hskip 3mm exist an element $c \in M_i \cap M_j$. By c), we have that $C_{i-1}\cup \{c\}$ meets all previously

\hskip 3mm claimed $k$-sets, then

\vskip 2mm

\hskip 1cm $\circ$ let $m_i:=C_{i-1}\cup \{c\} \cup N_i$
with $|N_i|=k-|C_{i-1}|-1$ and $N_i \cap (\cup_{j=0}^{i-1}m_j\cup\cup_{j=1}^iM_j)=\emptyset$;

\vskip 1mm

\hskip 1cm $\circ$ let $A_i:=A_{i-1}\setminus \{c\}, C_i:=C_{i-1}\cup \{c\}$.

\vskip 2mm

\hskip 1cm The set $m_i$ is a legal move for \textit{minimizer} in this subcase as $m_i$ contains $C_{i-1}$

\hskip 1cm  and $c$. Now observe that

\vskip 2mm

\hskip 1cm a) is satisfied as it was satisfied in step $i-1$ thus $A_i\subset A_{i-1}\subset m_0$ holds;

\vskip 1mm

\hskip 1cm b) is satisfied as it was satisfied in step $i-1$,  $A_{i-1} \cup C_{i-1} = A_i \cup C_i$ and by

\hskip 14mm the choice of $N_i$ and $c$;

\vskip 1mm

\hskip 1cm c) is satisfied as it was satisfied in step $i-1$, $A_{i-1} \cup C_{i-1} = A_i \cup C_i$ and

\hskip 14mm $c \in M_i \cap M_j$;

\vskip 1mm

\hskip 1cm d) is satisfied by the fact that that $C_{i-1}$ meets all previously claimed sets but $M_j$,

\hskip 14mm$c \in M_i \cap M_j$ and $C_i=C_{i-1}\cup \{c\}$.

\vskip 0.3truecm

Note that if step $i$ is in \textsc{Case II}, then step $i+1$ is in \textsc{Case I}.

\vskip 6mm

We have just seen that \textit{minimizer} is able to claim $k$-sets $m_1,m_2,...,m_k$ such that there exist sets $A_i, C_i$ ($2 \le i \le k$) satisfying the properties a)-d). The following claim states that in at least one third of the rounds \textit{minimizer} does not need to increase $C_i$ and thus obtains a small covering set.

\begin{claim}
\label{clm:ck}
For any $2 \le i \le k$, the inequality $|C_i| \le 3+\lfloor \frac{2(i-2)}{3}\rfloor$ holds.
\end{claim}

\begin{proof}[Proof of Claim]
Let $\alpha_i:=|\{j:2 < j \le i, |C_{j-1}|=|C_j|\}|$, i.e. the number of steps when we are in the first two subcases of Case 1. Let $\beta_i:=|\{j: 2<j\le i, |C_j|=|C_{j-1}|+1, j\text{th turn is in Case 1}\}|$ and $\gamma_i:=|\{j: 2<j\le i,  j\text{th turn is in Case 2}\}|$. Clearly, we have $\alpha_i+\beta_i+\gamma_i=i-2$. If the $j$th turn is in the last subcase of Case 1 or in Case 2, then $C_j$ meets all previously claimed $k$-subsets and $m_j$, as well. Thus we obtain $\gamma_i+\beta_i \le \alpha_i+\beta_i+1$ and therefore $\gamma_i \le \alpha_i+1$. Also, as in the last subcase of Case 1 the size of $A_j$ decreases, and this size only increases if we are in the first two subcases of Case 1, we obtain $\beta_i \le \alpha_i$. From these three inequalities it follows that $1+(i-2)/3 \le \alpha_i$ holds and thus statement of \clref{ck}.
\end{proof}

Let $M_{k+1}$ be the next move of \textit{Maximizer}. By property d), there can be at most one set $M_j$ that is disjoint from $C_k$. If \textit{minimizer} picks an element $m$ of $M_{k+1} \cap M_j$ and claims the $k$-set $m_{k+1}:=C_k \cup \{m\} \cup N_{k+1}$ with $|N_{k+1}|=k-|C_{k}|-1$ and $N_k \cap (\cup_{j=0}^{k}m_j\cup\cup_{j=1}^{k+1}M_j)=\emptyset$, then the set $C=C_k \cup \{m\}$ and $m_1,...,m_{k+1}$ satisfy the conditions of \pref{minmax}. This proves $\tau_m(n,k)\le |C| \le \lceil 2k/3\rceil +3$.

If \textit{Maximizer} starts the game, then \textit{minimizer} can imitate his previous strategy to obtain a sequence of moves  $M_1,m_1,M_2,m_2,...,M_k,m_k$ with the following slight modification: the sets $A_i$, $C_i$ and the moves $m_i$ still satisfy properties b) - d), but property a) is replaced with

\vspace{3mm}

a') $A_i \subseteq m_0$, $C_{i-1} \subseteq C_i$, $|C_i| \le |C_{i-1}|+1$.

\vspace{2mm}

The proof is identical to the one  when \textit{minimizer} starts the game. In this way, \textit{minimizer} obtains a $C_k$ with the same size as before and in his (\textit{k+1})st and (\textit{k+2})nd moves, he can add two more elements to obtain a set $C'$ that is just one larger and satisfies the conditions of \pref{minmax} together with $m_2,...,m_{k+1},m_{k+2}$. Therefore we obtain $\tau_M(n,k)\le |C'| \le \lceil 2k/3\rceil +4$.
\end{proof}

Now we turn our attention to the upper bound on $\tau_m(n,k)$ and $\tau_M(n,k)$. In the following proof we will use the following notations: the \textit{degree} of a vertex $x$ in a family $\cF$ of sets is $deg_{\cF}(x):=|\{F : x \in F \in \cF \}|$. Also, we will write $\cM_i:=\{M_j : j \le i \}$ for
 the family of $k$-sets that $\textit{Maximizer}$ picks until step $i$.

\begin{lemma}
\label{lem:taumax}
For any positive integer $k$, if $k^{3/2} \le n$, then $$ \frac{1}{2}\sqrt{k} \le \tau_m(n,k),\tau_M(n,k)$$ holds.
\end{lemma}

\begin{proof} Note that $$ \frac{|\cF|}{\max_{x \in X} deg_{\cF}(x)} \le \tau(\cF)$$ holds for any set $X$ and a family $\cF$ of subsets of $X$. Therefore a possible strategy for $\textit{Maximizer}$ is to keep $$\max_{s \in [n]}deg_{\cM_j}(s)$$ as small as possible. If he is able to do so long enough, then already the sets claimed by him will ensure that the covering number of the resulting family is large. We claim
that
\textit{Maximizer} can choose legal steps $M_1$, ..., $M_{\lfloor  k^{1/2}\rfloor}$ such that $$\max_{s \in [n]}deg_{\cM_{\lfloor k^{1/2}\rfloor}}(s) \le 2$$
holds.

In order to establish the aim above, in the $i$th step \textit{Maximizer} will choose his set $M_i$ with $M_i=M^1_i \cup M^2_i$, $M^1_i \cap M^2_i =\emptyset$, $|M^1_i|=\lfloor  k^{1/2}\rfloor-1=:l$ and $|M^2_i|=k-l+1$. The $M^1_i$'s are independent of how \textit{minimizer} picks his sets, they are chosen to ensure that $M_j \cap M_i\neq \emptyset$ holds for any pair $1 \le j<i \le l+1$. The other part $M^2_i$ is supposed to ensure that $M_i$ meets all $i$ or $i-1$ sets that \textit{minimizer} has claimed by that point of the game (depending on who started the game).

We define the $M^1_i$'s inductively: let $M^1_1=[l]$ and assume the elements of $M^1_1,...,M^1_{i-1}$ are enumerated increasingly as $v^1_1,...v^1_l,v^2_1,...,v^2_l,...,v^{i-1}_1,...,v^{i-1}_l$, then let
\[
M^1_i=\{v^1_{i-1},v^2_{i-1},...,v^{i-1}_{i-1}\} \cup \{u_i,u_i+1,...,u_i+l-i\}
\]
where $u_i:=1+\sum_{h=0}^{i-1}l-h$. By definition, $M^1_i$ meets all previous $M^1_j$'s in exactly one point and all intersection points are different, thus we obtain that the maximum degree is 2. Also, since the $M^1_i$ introduces $l-i+1$ new points, we have $U:=\bigcup_{i=1}^{l+1} M^1_i=\left[\frac{l(l+1)}{2}\right]$ and $\frac{l(l+1)}{2} \le k/2$.

We still have to show that \textit{Maximizer} can define the $M^2_i$'s such that $M^2_i$ intersect all previously claimed sets of \textit{minimizer} and the maximum degree is kept at most 2. \textit{Maximizer} tries to pick the $M^2_i$'s such that the following three properties
hold for $i \le \lfloor  k^{1/2}\rfloor$ with the notation $\cM^2_i:=\{M^2_j : j \le i \}$:

\vspace{2mm}

$(1)$ $|\{x : deg_{\cM^2_i}(x) = 2 \}| \le \frac{i^2}{2}$,

\vspace{1mm}

$(2)$ $\{x : deg_{\cM^2_i}(x) \ge 3 \} = \emptyset$, and

\vspace{1mm}

$(3)$ $U \cap M^2_i=\emptyset$.

\vspace{3mm}

We prove by induction on $i$ that he can choose $M^2_i$ satisfying $(1)$, $(2)$, and $(3)$. $M^2_1$ can be chosen arbitrarily with the restriction that it is disjoint from $U$ and if \textit{minimizer} starts the game, then it should meet $m_1$. Note that the latter is possible as $|U| \le k/2$ and thus $|m_1 \setminus U| \ge k/2$ holds.

Assume \textit{Maximizer} was able to pick $M^2_1,...,M^2_{i-1}$ for some $1<i\le \lfloor k^{1/2} \rfloor$ satisfying $(1)$, $(2)$, and $(3)$ and now he has to pick $M^2_i$. Observe that by the inductive
hypothesis for all $1 \le h <i$ we have $$|m_h \setminus (\{x : deg_{\cM^2_{i-1}}(x) = 2 \} \cup U)| > k-k/2-(i-1)^2\ge 2 k^{1/2},$$
and the sets $M^2_h\setminus \{x : deg_{\cM^2_{i-1}}(x) = 2 \} $ with $1 \le h <i$ are pairwise disjoint. Thus if \textit{Maximizer} picks $M^2_i$ such that it is disjoint from $\{x : deg_{\cM^2_{i-1}}(x) \ge 2 \} \cup U$, then $(2)$ and $(3)$ are clearly satisfied. Let us fix $x_h \in m_h \setminus (\{x : deg_{\cM^2_{i-1}}(x) \ge 2 \} \cup U)$ for all $1 \le h <i$ and let $M^2_i=\{x_h: 1\le h < i\} \cup M$ where $|M|=k-l+1-|\{x_h: 1\le h < i\}|$ and $M \cap \bigcup_{h=1}^{i-1}(m_h\cup M_h)=\emptyset$. It is possible to satisfy this latter condition as $n \ge k^{3/2}$. We see that the new degree-2 elements are the $x_h$'s and thus there is at most $i-1$ of them.
As $(i-1)^2 +i-1 \le i^2$ we know that $\cM_i$ satisfies $(1)$.

This inductive construction showed that the maximum degree of $\cM_{l+1}$ is 2 and thus its covering number is at least $(l+1)/2=\lfloor  \frac{\sqrt{k}}{2}\rfloor$.
\end{proof}

\vskip 0.3truecm

\begin{proof}[Proof of \tref{main}.] To obtain the upper bound, Fast can use \textit{Maximizer}'s strategy in the $\tau$-game, that is whenever it is his turn, he just copies whatever \textit{Maximizer} would do in the same situation. According to \lref{taumax}, Fast can make sure that the resulting maximal intersecting family $\cF$ will have covering number at least $\frac{\sqrt{k}}{2}$. Thus by \pref{tau}, we have
\[|\cF| \le k^{\sqrt{k}/2}\binom{n-k^{\sqrt{k}/2}}{k-k^{\sqrt{k}/2}}=O_k(n^{k-\sqrt{k}/2}).
\]

To obtain the lower bound, Slow can imitate \textit{minimizer}'s strategy to ensure that, by \lref{taumin}, the resulting maximal intersecting family $\cF$ will have covering number at most $\lceil 2k/3\rceil +2$. Thus, by \pref{tau}, we have $|\cF| \ge \binom{n-\lceil 2k/3\rceil-4}{k-\lceil2k/3\rceil -4}=\Omega_k(n^{k/3-5})$.
\end{proof}

\section{Concluding remarks and open problems}

The main result of the present paper, \tref{main} states that the exponent of $gsat(\mathbb{I}_{n,k})$ grows linearly in $k$. We proved that the constant of the linear term is at least $1/3$ and at most 1. We deduced this result by obtaining lower and upper bounds on the covering number of the family of sets claimed during the game and using a well-known relation between the covering number and the size of an intersecting family. However we were not able to determine the order of magnitude of the covering number. We conjecture this to be linear in $k$. If it is true, this would have the following consequence.

\begin{conjecture}
There exists a constant $c>0$ such that for any $k\ge 2$ and $n \ge n_0(k)$ the inequality $gsat_F(\mathbb{I}_{n,k}), gsat_S(\mathbb{I}_{n,k}) \le O(n^{(1-c)k})$ holds.
\end{conjecture}

In the proof of \lref{taumax}, there are two reasons for which \textit{Maximizer} cannot continue his strategy for more than $\sqrt k$ steps. First of all the union of the $M^1_j$'s becomes too large, and second of all the intersection points of the $M^2_j$'s and the sets of \textit{minimizer} should be disjoint. The first problem can probably be overcome by a result of Kahn \cite{K} who showed the existence of an $r$-uniform intersecting family $\cF_r$ with $|\cF|= O(r)$, $|\bigcup_{F \in \cF_r}F|=O(r)$ and $\tau(\cF_r)=r$.

As we mentioned in the introduction, the answers to game saturation problems considered so far did not depend on which player makes the first move, while this is the case for intersecting families if $k=2$. However, if we consider the $\tau$-game, both our lower and upper bounds differ by at most one depending on whether it is \textit{Maximizer} or \textit{minimizer} to make the first move. Thus we formulate the following conjecture.
\begin{conjecture}
There exists a constant $c$ such that $|\tau_M(n,k) -\tau_m(n,k)| \le c$ holds independently of $n$ and $k$.
\end{conjecture}

Intersecting families are most probably one of the two most studied classes of families in extremal set system theory. The other class is that of \textit{Sperner families}: families $\cF$ that do not contain two different sets $F,F'$ with $F \subset F'$. The downset $\mathbb{S}_n:=\{\cF \subseteq 2^{[n]}: \cF \hskip 0.2truecm \text{is Sperner} \}$ is another example for which the two game saturation numbers differ a lot. Clearly, if Fast starts the game, then he can claim either the empty set or $[n]$ to finish the game immediately as both $\{\emptyset\}$ and $\{[n]\}$ are maximal Sperner families in $2^{[n]}$, thus we have $gsat_F(\mathbb{S}_n)=1$. It is not very hard to see that if Slow starts with claiming a set $F \subset [n]$ of size $\lfloor n/2 \rfloor$, then the game will last at least a linear number of turns. Indeed, consider the family $\cN_F=\{F\setminus \{x\} \cup \{y\}: x \in F, y \in [n]\setminus F\}$. $\cN_F$ has size about $n^2/4$, while
\[
\max_{G: \hskip 0.1truecm G \not\subset F, F \not\subset G}|\{F' \in \cN_F: F' \subseteq G \hskip 0.2truecm \text{or}\ G \subseteq F'\}|=\lceil n/2 \rceil.
\]
This shows that $gsat_S(\mathbb{S}_n) \ge n/2$. One can improve this bound, but we were not able to obtain a superpolynomial lower bound nor an upper bound $o(\binom{[n]}{\lfloor n/2 \rfloor})=o(ex(\mathbb{S}_n))$.

\end{document}